\documentclass[11pt]{amsart}
\usepackage{amsmath, amssymb,amsthm, fullpage, epic, eepic, graphics, mathrsfs, hyperref}
\usepackage{pgfplots}
\usepackage{subcaption} 
\usepackage{caption}

\newtheorem{thm}[equation]{Theorem}
\newtheorem{prop}[equation]{Proposition}

\newtheorem{lemma}[equation]{Lemma}

\theoremstyle{definition}
\newtheorem{defn}[equation]{Definition}

\theoremstyle{remark}

\makeatletter

\begin{document}

\title{A Bollobás-type problem: from root systems to Erd\H{o}s--Ko--Rado}

\author{Patrick J. Browne, Qëndrim R. Gashi and Padraig \'O Cath\'ain}
\thanks{P. J. Browne and P. \'O Cath\'ain acknowledge the National Forum for Enhancement of Teaching and Learning and Technical University of the Shannon who supported this publication through \textit{Strategic Alignment of Teaching and Learning Enhancement} funding. P. \'O Cath\'ain acknowledges the support of the Conference Participation Scheme of the Faculty of Humanities and Social Sciences at Dublin City University.}
\address{Department of Mathematics,
University of Prishtina\\
rr. Nena Tereze, p.n., 10000 Pristina,
Kosovo\\
qendrim.gashi@uni-pr.edu}

\address{Department of Electrical \& Electronic Engineering, Technological University of the Shannon, Limerick, V94 EC5T, Ireland, patrick.browne@tus.ie  }

\address{Fiontar agus Scoil na Gaeilge, Dublin City University, All Hallows Campus, D09 N920, Ireland, padraig.ocathain@dcu.ie }

\begin{abstract}
Motivated by an Erd\H{o}s--Ko--Rado type problem on sets of strongly orthogonal roots in the $A_{\ell}$ root system, we estimate bounds for the size of a family of pairs $(A_{i}, B_{i})$ of $k$-subsets in $\{ 1, 2, \ldots, n\}$ such that $A_{i} \cap B_{j}= \emptyset$ and $|A_{i} \cap A_{j}| + |B_{i} \cap B_{j}| = k$ for all $i \neq j$. This is reminiscent of a classic problem of Bollob\'as. We provide upper and lower bounds for this problem, relying on classical results of extremal combinatorics and an explicit construction using the incidence matrix of a finite projective plane.

\textbf{MSC 2020: 
05D05, 
17B22 
}\\
\textbf{Keywords: Erd\H{o}s-Ko-Rado, root system, strongly orthogonal roots}
\end{abstract}

\maketitle

\section{Introduction}
Given a vector space $V \leq \mathbb{R}^{\ell+1}$, equipped with the standard inner product, a subset $R\subseteq V$ is called a \textit{root system} if the following conditions are satisfied.
\begin{enumerate} 
\item The roots span $V$, and do not contain the zero vector.
\item The only scalar multiple of $\alpha \in R$ also contained in $R$ is $-\alpha$. 
\item Each root $\alpha \in R$ determines a reflection $x \mapsto x - 2\frac{\langle x,\alpha\rangle}{\langle \alpha, \alpha\rangle} \alpha$. The set $R$ is closed under these reflections. 
\item For any $\alpha, \beta \in R$, the value $2\frac{\langle \alpha, \beta\rangle}{\langle \alpha, \alpha\rangle}$ is an integer.
\end{enumerate}

It is a remarkable fact that there are precisely four infinite families of irreducible root systems, together with additional examples in low dimensions. Root systems play a key role in the classification of semi-simple Lie algebras, see Chapter 9 of Humphrey's monograph \cite{humphreys}.

\begin{defn}\label{TypeA}
Denote by $\varepsilon_{i}$ the $i^{\textrm{th}}$ standard basis vector for $\mathbb{R}^{\ell+1}$. The set of vectors $\varepsilon_{i} - \varepsilon_{j}$ for $j\neq i$ is the \textit{Type A root system}, denoted $A_{\ell}$. The vectors $\alpha_{i} = \varepsilon_{i} - \varepsilon_{i+1}$ with $i = 1, 2, \ldots, \ell$ form a \textit{basis} for $A_{\ell}$; these are called \textit{simple roots}. 
\end{defn}

It is a simple calculation to verify that the vectors of Definition \ref{TypeA} indeed satisfy the axioms of a root system. The inner product of two roots is negative if and only if their sum is a root; similarly, the inner product of two roots is positive if and only if their difference is a root. The Dynkin diagram for the $A_\ell$ root system is given below. This shows the $\ell$ simple roots with an edge between them if and only if their sum is a root. 

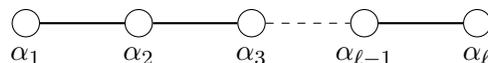
\begin{figure}[h]
\begin{center}
 \begin{tikzpicture}
\node[shape=circle,draw,label=below:$\alpha_1$] (1) at (0,8)   {};
\node[shape=circle,draw,label=below:$\alpha_2$] (2) at (1.5,8)   {};
\node[shape=circle,draw,label=below:$\alpha_3$] (3) at (3,8)   {};
\node[shape=circle,draw,label=below:$\alpha_{\ell -1}$] (4) at (4.5,8)   {};
\node[shape=circle,draw,label=below:$\alpha_\ell$] (5) at (6,8)   {};
\path
(1) edge[thick]  (2)
(2) edge[thick] (3)
(3) edge[dashed] (4)
(4) edge[thick] (5);
\end{tikzpicture}
\end{center}
\caption{Dynkin diagram of $A_\ell$.}
\label{da3}
\end{figure}

\begin{defn}
Two roots in $R$ are \textit{strongly orthogonal} if neither their sum nor their difference is a root. A subset of $R$ consisting of $k$ roots, each pair of which is strongly orthogonal is called a \textit{strongly orthogonal subset} of $R$. The set of all such subsets is denoted $SOS_{k}(R)$. 
\end{defn}

For example, a set of simple roots in $A_{\ell}$ is strongly orthogonal if and only if they correspond to an independent set in the Dynkin diagram (that is, a subset of vertices in the diagram, no two connected by an edge). Recall that the \textit{support} of a vector is the collection of co-ordinates in which the vector is non-zero. An arbitrary pair of roots in $A_{\ell}$ is strongly orthogonal if and only if the roots have disjoint supports. While this paper is mostly concerned with the type $A$ root systems, much is known about $SOS_{k}(R)$ for an arbitrary root system. The next result characterises the maximal size of a strongly orthogonal subset in an irreducible root system. 

\begin{prop}[\cite{agaoka}, \cite{bedullia}] \label{known} 
The set $SOS_k(R)$ is non-empty in the following cases:\smallbreak
\begin{minipage}{.5\textwidth}
\begin{itemize}
\item[(i)] For $k \leq \left\lfloor \frac{\ell+1}{2} \right\rfloor $ when $R=A_{\ell}$.
\item[(ii)] For $k \leq \ell$ when $R=B_{\ell}$.
\item[(iii)] For $k \leq \ell$ when $R=C_{\ell}$.
\item[(iv)] For $k \leq 2 \left\lfloor \frac{\ell}{2} \right\rfloor$ when $R=D_{\ell}$.
\end{itemize}
\end{minipage}
\hfill  
\begin{minipage}{.5\textwidth}
\begin{itemize}
\item[(v)] For $k \leq 3$ when $R=F_4$.
\item[(vi)] For $k \leq 4$ when $R=E_6$.
\item[(vii)] For $k \leq 7$ when $R=E_7$.
\item[(viii)] For $k \leq 8$ when $R=E_8$.
\item[(ix)] For $k \leq 2$ when $R=G_2$.
\end{itemize}
\end{minipage}
In all other cases, $SOS_k(R)$ is empty.
\end{prop}

\section{Erd\H{o}s--Ko--Rado type theorems}

Let $n, k \in \mathbb{N}$ and write $[n]:=\{ 1, 2, \ldots, n\}$. Denote by $\mathcal{V}(n,k):=\{ X \subseteq [n]: |X|=k \}$ the set of $k$-element subsets of $[n]$. The classical Erd\H{o}s--Ko--Rado theorem is a cornerstone of extremal combinatorics. 

\begin{thm}[\cite{EKR-original,EKR}]\label{EKR}
Let $\mathcal{F} \subseteq \mathcal{V}(n,k)$ and suppose that $X \cap Y \neq \emptyset$ for all $X, Y \in \mathcal{F}$. Then 
\[ |\mathcal{F}| \leq \binom{n-1}{k-1}\,.\]
Moreover, if $n > 2k$, equality holds if and only if $\mathcal{F}$ consists of all elements of $\mathcal{V}(n,k)$ that contain a given element from $[n]$.
\end{thm}

Equivalently, let $\Gamma(n,k)$ be the graph which has as vertex set $\mathcal{V}(n,k)$, with an edge between two $k$-subsets when they have non-trivial intersection. The first half of Theorem \ref{EKR} gives an upper bound on the size of a maximal clique in $\Gamma(n,k)$ and the second half characterises the maximal cliques precisely. Motivated by Tur\'an's problem in graph theory, Bollob\'as proved a lemma on intersecting set systems; it is notable that the size of the underlying set does not appear in the bound. This statement of the result is due to Lov\'asz.

\begin{thm}[\cite{bollobas,lovasz}]
Let $X_1, \ldots, X_N \in \mathcal{V}(n,k_{1})$ and 
$Y_1, \ldots, Y_N \in \mathcal{V}(n,k_{2})$ be distinct. 
If for each $i,j = 1, \ldots, N$,
\[ X_i \cap Y_i = \emptyset, \;\;\; X_i \cap Y_j \neq \emptyset \]
then $N \leq \binom{k_{1}+k_{2}}{k_{1}}$.
\end{thm}

The obvious configuration of subsets here consists of partitions of a set of size $k_{1} + k_{2}$ into complementary parts of sizes $k_{1}$ and $k_{2}$. The content of the theorem is that larger configurations satisfying these intersection properties cannot exist, even if the size of the underlying set is increased. 
Our main result may be stated as a result on the size of a family of sets satisfying hypotheses similar to those of Bollob\'as.

\begin{thm}\label{main}
Let $X_{1}, \ldots, X_{N}$ and $Y_{1}, \ldots, Y_{N} \subseteq \mathcal{V}(n,k)$ be such that $X_i \cap Y_j = \emptyset$ for all $1 \leq i, j \leq N$. Suppose that for all $i \neq j$,
\[|X_i \cap X_j|+|Y_i \cap Y_j|=k \,.\]

Then $N \leq n$ for all $k$. If $n > k4^{k}$, then $N \leq \lfloor \frac{n+2-k}{k} \rfloor$, and the sets attaining this bound may be constructed explicitly.
\end{thm}

We will prove this theorem by restating the problem in terms of strongly orthogonal sets of roots in the $A_{\ell}$ root system. 

\begin{defn}\label{SOS-clique}
Let $\Gamma \in SOS_{k}(R),$ we write $|\Gamma|  = \sum_{\gamma \in \Gamma} \gamma$ for the sum of the roots in $\Gamma$. A subset $\mathcal{F} \subseteq SOS_{k}(R)$ is a \textit{SOS-clique} 
    if and only if for every $\Gamma_{i}, \Gamma_{j} \in \mathcal{F}$ there exists some $\Gamma_{i,j} \in \mathcal{F}$ such that 
\begin{equation}\label{PropS}
|\Gamma_i|-|\Gamma_j|= |\Gamma_{i,j}|\,.
\end{equation}
The maximal size of an SOS-clique in $SOS_{k}(R)$ will be denoted $\mu_{k}(R)$.
\end{defn}

A set of strongly orthogonal roots $\Gamma \in SOS_k(A_{\ell})$ corresponds to a pair of $k$-subsets as follows: $X_{i}$ consists of the coordinates where $|\Gamma|$ is positive, and $Y_{i}$ to the coordinates where $|\Gamma|$ is negative. The condition on intersections of Theorem \ref{main} is clearly equivalent to that of Equation \eqref{PropS}, hence the upper and lower bounds on $\mu_{k}(A_{\ell})$ given in the next section will yield a proof of Theorem \ref{main}. 

This is a problem of Erd\H{o}s--Ko--Rado type: construct a graph in which vertices are labelled by the vectors $|\Gamma|$ for $\Gamma \in SOS_{k}(R)$; with an edge between $|\Gamma_{i}|$ and $|\Gamma_{j}|$ if and only if the difference of their vectors is again the label of a vertex in the graph. The SOS-clique of Definition \ref{SOS-clique} is a clique in this graph. As in the Erd\H{o}s--Ko--Rado theorem, we give both an upper bound on the size of maximal clique and a characterisation of all maximal cliques when $\ell$ is sufficiently large in terms of $k$.

\section{Main Result}


Recall a pair of roots is strongly orthogonal if and only if their supports are disjoint, and that the set $SOS_{k}(A_{\ell})$ consists of $k$-tuples of pairwise strongly orthogonal roots. It will be convenient to refer to an explicit element of $SOS_{k}(A_{\ell})$ on occasion. Define $\beta_{j} = \sum_{i=1}^{k} \alpha_{i+j}$, then $\Gamma_{1} = \{ \beta_{j} : j =0, \ldots, k-1\}$ is such a set. The sum $|\Gamma_{1}| = \sum_{j=0}^{k-1} \beta_{j}$ is a vector with the first $k$ entries equal to $+1$ followed by $k$ entries equal to $-1$. Let $\mathcal{F}$ be an SOS-clique as in Definition \ref{SOS-clique}. We assume without loss of generality that $\Gamma_{1} \in \mathcal{F}$. For any $\Gamma \in \mathcal{F}$ distinct from $\Gamma_{1}$, the supports of $\Gamma$ and $\Gamma_{1}$ intersect in precisely $k$ co-ordinates. We write $S(\Gamma,\Gamma_{1})$ for this set.

\begin{defn}\label{sunflowerDef} 
Let $\mathcal{F} \subseteq SOS_{k}(A_{\ell})$ be an SOS-clique. If for all $\Gamma_{i}, \Gamma_{j} \in \mathcal{F}$, the sets $S(\Gamma_{i}, \Gamma_{j})$ are equal, we say that $\mathcal{F}$ is a \textit{sunflower}.
\end{defn}

\begin{lemma}\label{TypeI}
Suppose that $\mathcal{F} \subseteq SOS_{k}(A_{\ell})$ is a sunflower. Then $|\mathcal{F}| \leq \lfloor \frac{\ell + 1}{k}\rfloor -1$. 
\end{lemma} 

\begin{proof} 
Let $\mathcal{F} = \{ \Gamma_{1}, \Gamma_{2}, \ldots, \Gamma_{m} \}$. Denote by $X$ the set of column indices in which $|\Gamma_{i}|$ and $|\Gamma_{j}|$ agree. This is a set of size $k$, and it is immediate from Definition \ref{SOS-clique} that the remaining non-zero entries of $|\Gamma_{i}|$ are the only non-zero entries in their respective columns. Hence there exists disjoint $k$-sets $X_{1}, \ldots, X_{m}$ such that the support of $|\Gamma_{i}| = X \cup X_{i}$. Hence $\ell+1$, the total number of columns, is at least $k|\mathcal{F}| + k$.
\end{proof}

Next we show that the number of vectors which do not intersect $\Gamma_{1}$ in a fixed subset of size $k$ is bounded by a function of $k$.

\begin{lemma}\label{TypeII}
Let $\mathcal{F} \subseteq SOS_{k}(A_{\ell})$ be an SOS-clique. Suppose that $S(\Gamma_{i}, \Gamma_{j}) = S(\Gamma_{x}, \Gamma_{y})$ if and only if  $\{ \Gamma_{i}, \Gamma_{j} \} = \{ \Gamma_{x}, \Gamma_{y} \}$. Then $|\mathcal{F}| \leq \binom{2k}{k} +1$.
\end{lemma} 

\begin{proof} 
Consider a fixed $\Gamma_{1} \in \mathcal{F}$. By Definition \ref{SOS-clique}, the support of every other $\Gamma_{i} \in \mathcal{F}$ intersects the support of $\Gamma_{1}$ in $k$ columns. By hypothesis, these $k$-sets are all distinct. Hence $|\mathcal{F}| \leq \binom{2k}{k}+1$, by the pigeonhole principle. 
\end{proof}

We recall a classical result of Ray-Chaudhuri and Wilson, and then we will be in a position to prove our main theorem. 

\begin{thm}[\cite{RCW}] \label{RCW}
Let $\mathcal{F}$ be a family of $k$-subsets of $[n]$, and let $S \subseteq [k-1]$ be of size $s$. Suppose that all pairwise intersections of elements of $\mathcal{F}$ have size in $S$. Then $|\mathcal{F}| \leq \binom{n}{s}$. 
\end{thm}

\begin{thm}\label{TheoremA}
For any $k \in \mathbb{N}$, we have $\mu_{k}(A_{\ell}) \leq \ell+1$. 
For $\ell > k4^{k}$,  
\[ \mu_k(A_{\ell}) = \left\lfloor \frac{{\ell}+1-k}{k} \right\rfloor\,.\]
If $\ell > k4^{k}$ and $|\mathcal{F}| = \mu_{k}(A_{\ell})$ then $\mathcal{F}$ is a sunflower. 
\end{thm}

\begin{proof}
Let $\Gamma_{1}, \Gamma_{2}, \ldots, \Gamma_{m} \in \mathcal{F}$ be a maximal sunflower of $\mathcal{F}$. As in the proof of Lemma \ref{TypeI}, there exist disjoint sets $X, X_{1}, \ldots, X_{m}$ of coordinates, each of size $k$, such that the support of $\Gamma_{i}$ is $X \cup X_{i}$. If there is no vector satisfying the condition of Lemma \ref{TypeII} then $|\mathcal{F}| = m+1$. Since $\ell+1 \geq k(m+1)$, the conclusion holds. 

Now suppose that there exists $\Gamma_{m+1} \in \mathcal{F}$ such that $S(\Gamma_{1}, \Gamma_{m+1}) \neq X$. The support of $\Gamma_{m+1}$ intersects $X \cup X_{i}$ in $k$ points, and so must intersect each $X_{i}$ non-trivially. Hence $m-1 \leq k$. Counting vectors by their relation with $\Gamma_{1}$, forming a sunflower with $\Gamma_{1}$. There are at most $\binom{2k}{k}-1$ possible intersections with the support of $\Gamma_{1}$ distinct from $X$. For each such intersection, there are at most $k$ vectors sharing that support. Hence $|\mathcal{F}| \leq k\binom{2k}{k} + 1$. Thus the hypothesis that $\mathcal{F}$ is not a sunflower leads to an upper bound on $|\mathcal{F}|$ which is independent of $\ell$. For fixed $k$, once $|\mathcal{F}| > k\binom{2k}{k} + 1$, all SOS-cliques are sunflowers. 

By Theorem \ref{RCW}, the number of $2k$-sets such that all pairwise intersections have size $k$ is bounded by $\ell+1$, for any choice of $\ell$ and $k$. This gives the general upper bound. A standard bound on central binomial coefficients gives $\binom{2k}{k} \leq 4^{k} -1$. Hence, when $|\mathcal{F}| \geq k 4^{k}$ then $\mathcal{F}$ is necessarily a sunflower. Both the stronger bound $|\mathcal{F}|\leq \left\lfloor \frac{{\ell}+1-k}{k} \right\rfloor$ and the characterisation of SOS-cliques meeting the bound follow.  
\end{proof} 

In the interest of presenting an accessible proof, we made no attempt to optimise the constant $k4^{k}$ in the proof of Theorem \ref{TheoremA}. It is likely that this can be improved with a more careful analysis. Trivially, $\mu_{1}(A_{\ell}) = \ell$ for any $\ell \in \mathbb{N}$. The next result evaluates $\mu_{2}(A_{\ell})$ precisely. 

\begin{thm}\label{smallA}
The sequence $\mu_{2}(A_{\ell})$ is $0,0,1,1,3$ for $1 \leq \ell \leq 5$. For $6 \leq \ell \leq 13$, it is equal to $6$ and for $\ell \geq 13$, its value is $\lfloor (\ell-1)/2 \rfloor$.
\end{thm}

\begin{proof}
Suppose that $\mathcal{F} \subseteq SOS_{2}(A_{\ell})$ is of maximal size. If $\mathcal{F}$ is a sunflower then $|\mathcal{F}| = \lfloor (l-1)/2 \rfloor$. We will show that otherwise, $|\mathcal{F}| \leq 6$. By the argument of Theorem \ref{TheoremA}, if $\mathcal{F}$ is not a sunflower, then it contains a sunflower of size at most $3$. 

If $\mathcal{F}$ contains a sunflower of three vectors, then there are $\ell+1 = 8$ columns. The following example shows that $6$ vectors are possible in this case. 
\[\begin{matrix}
    +&-& +&-& 0&0& 0&0\\
    +&-& 0&0& +&-& 0&0\\
    +&-& 0&0& 0&0& +&-\\
    +&0 &+&0& 0&-& 0&-\\
    +&0 &0&-& +&0& 0&-\\
    +&0 &0&-& 0&-& +&0
  \end{matrix}\]
No seventh vector is possible. The sunflower is forced; if these vectors intersected on two entries of the same sign a fourth vector would be impossible without forcing all vectors to satisfy the conditions of Lemma \ref{TypeI}. Each subsequent row is determined by the sequence of signs of its entries. These are sequence of length $4$ which contain two $1$'s and two $-1$'s and are mutually orthogonal as vectors in $\mathbb{R}^{4}$. It is easily verified that there are at most three such, and computation by hand shows that the configuration above is essentially unique. 

Next, suppose that at most two vectors intersect on any pair of columns. The example of Figure \ref{fig:fano} shows that there exist six vectors on seven columns with this property. By Definition \ref{SOS-clique}, every pair of vectors intersects in some pair of columns. Thus there can be at most $7$ vectors in total: there are at most $\binom{4}{2}$ pairs of columns in which a vector can intersect a fixed vector, and each intersection is unique. In a configuration with seven vectors, every pair of vectors intersects in two columns; and every pair of columns in the support of a vector is shared with a unique vector. While it is possible to find $7$ subsets of an $8$-set intersecting pairwise in a set of size $2$, as demonstrated below, there is no consistent way to choose columns of negative entries such that every row has two negative entries. 
\[\begin{matrix}
    1&1&1&1& 0&0&0&0\\
    1&1&0&0& 1&1&0&0\\
    0&0&1&1& 1&1&0&0\\
    1&0&1&0& 1&0&1&0\\
    1&0&0&1& 1&0&0&1\\
    0&1&0&1& 1&0&1&0\\
    0&1&1&0& 1&0&0&1\\
  \end{matrix}\]
If the $0$ entries in the above matrix are replaced by $-1$, these are the non-constant rows of a Hadamard matrix of order $8$. The uniqueness of the Hadamard matrix of order $8$ up to equivalence shows that $\mu_{2}(A_{7}) = 6$. 

Observe that rows $1,2,3$ of the displayed matrix are forced by the intersection condition, and that every subsequent vector must have a single non-zero entry in each of the first three pairs of columns. This leaves only one additional non-zero entry in each subsequent column, and a hand computation verifies that the remaining entries must be placed in two columns to complete the array. 

Thus, if $\mathcal{F} \subseteq SOS_{2}(A_{\ell})$ is not a sunflower then $|\mathcal{F}| \leq 6$. An easy hand computation gives the maximal size of an SOS-clique for $\ell \leq 5$, while sunflowers are of maximal size when $\ell \geq 13$; furthermore every maximal clique is a sunflower when $\ell \geq 15$. This completes the proof. 
\end{proof} 

Note that the above two results first appeared in a preprint which also contains further results on other types of root systems \cite{Gashi}. However, relations with Erd\H{o}s--Ko--Rado theory were not discussed in that work.

\section{Relation to finite projective planes}
In this section we use finite projective planes to construct large SOS-cliques. Recall 
that a finite projective plane of order $q$ is an incidence structure consisting of points and lines which satisfies the following properties: 
\begin{itemize}
    \item Every line contains $q+1$ points, and each point is incident with $q+1$ lines.
    \item Any two distinct lines meet at exactly one point, and any two distinct points lie on a unique line. 
    \item There exist four points, no three lying on a line. 
\end{itemize}
It is an elementary exercise to see that there are $q^2+q+1$ points and the same number of lines in a projective plane of order $q$. The incidence matrix $N$ for such a plane is a square matrix of size $q^2+q+1$, in which columns are labelled by points, and rows by lines. An entry is $+1$ when a point is incident with a line and $0$ otherwise. The axioms force $NN^{\top} = qI + J$ where $J$ is the all-ones matrix. 

\begin{prop} \label{ProjProp}
Let $r_{1}, r_{2}, \ldots, r_{n}$ be the rows of the incidence matrix of a projective plane. 
Then the vectors $r_{1} - r_{i}$ for $2\leq i\leq n$ yield an SOS-clique of size $q^{2} + q$ in $A_{q^{2} + q}$.
\end{prop} 

\begin{proof} 
There is a unique column in which vectors $r_{1}$ and $r_{i}$ are both equal to $1$, and each vector has $q+1$ entries $1$. Thus $r_{1} - r_{i}$ has $q$ entries $1$ and $q$ entries equal to $-1$. 

A second vector $r_{1} - r_{j}$ agrees with $r_{1} - r_{i}$ in precisely $q$ columns. There are two possibilities: if the vectors $r_{1}, r_{i}, r_{j}$ all share a $1$ in the same column, then $r_{1} - r_{i}$ and $r_{1} - r_{j}$ do not agree in any negative entry, but agree in $q$ positive entries. Otherwise, $r_{i}$ and $r_{j}$ share a non-zero entry in a column disjoint from the support of $r_{1}$, and each vector intersects $r_{1}$ in a single (distinct) column. 
Hence there are $q-1$ columns in which $r_{1}-r_{i}$ and $r_{1}-r_{j}$ share an entry $+1$ and a unique column in which they have an entry $-1$. In every case, the vectors agree in $q$ columns. These vectors can be trivially decomposed into sums of $k$ roots of the $A_{q^{2} + q}$ root system, and so yield an $SOS_{q}(A_{q^{2} + q})$. 
\end{proof} 

As an example consider the projective plane of order $2$ (Fano plane). Figure \ref{fig:fano} shows the differences between the first and the $6$ subsequent rows of this incidence matrix. Note that each pair of vectors intersects in two columns, but they do not all intersect in the same pair of columns; and so do not satisfy the conditions of Lemma \ref{TypeI}. 

\begin{figure}[h]
\centering
\begin{subfigure}{0.49\linewidth} 
\centering
  $\begin{pmatrix}
    1&1&1&0&0&0&0\\
    1&0&0&1&1&0&0\\
    1&0&0&0&0&1&1\\
    0&1&0&1&0&1&0\\
    0&1&0&0&1&0&1\\
    0&0&1&1&0&0&1\\
    0&0&1&0&1&1&0
  \end{pmatrix}$
  \label{fig:subfigure1}
\end{subfigure}
\begin{subfigure}{0.49\linewidth} 
\centering
  $\begin{matrix}
    0&+&+&-&-&0&0\\
    0&+&+&0&0&-&-\\
    +&0&+&-&0&-&0\\
    +&0&+&0&-&0&-\\
    +&+&0&-&0&0&-\\
    +&+&0&0&-&-&0
  \end{matrix}$
  \label{fig:subfigure2}
\end{subfigure}
\caption{Incidence matrix for the Fano plane, shown next to the first row minus subsequent rows.} \label{fig:fano}
\end{figure}

For a prime power $q$, Proposition \ref{ProjProp} gives an SOS-clique of size $q^{2} + q$ in $SOS_{q}(A_{q^{2} + q})$. On the other hand, the bound of Theorem \ref{TheoremA} is $q^{2} + q + 1$. Hence our bound is within $1$ of optimality infinitely often. The following questions are natural. 
\begin{itemize} 
     \item \textbf{Question 1:} Is the bound of Theorem \ref{TheoremA} ever met with equality? 
     
     For fixed $k$, what is maximal size of an SOS-clique, $\mu_{k}(A_{\ell})$, when $k \leq \ell \leq k4^{k}$? 
    \item \textbf{Question 2:} For fixed $k$, what is the smallest $C$ such that all maximal SOS-cliques are sunflowers for $\ell > C$? 
    \item \textbf{Question 3:} Do there exist constants $c, C$ such that for every $\ell > C$, an SOS-clique of size at least $\ell - c$ comes from a projective plane?
    \end{itemize}
The corresponding questions for other root systems also remain open, and will be the subject of a future investigation by the authors. 

\bibliographystyle{plain}
\bibliography{references}
\end{document}